\documentclass[11pt]{amsart}
\usepackage{amsmath,amsfonts,amsthm}
\usepackage{amssymb,latexsym}
\usepackage[colorlinks]{hyperref}
\usepackage{graphicx}
\usepackage[all]{xy}
\usepackage{layout}
\theoremstyle{plain}
\newtheorem{theorem}{Theorem}[section]
\newtheorem*{theorem*}{Theorem}
\newtheorem{lemma}[theorem]{Lemma}

\newtheorem{corollary}[theorem]{Corollary}

\theoremstyle{definition}

\newtheorem{definition}[theorem]{Definition}

\theoremstyle{remark}

\numberwithin{equation}{section}

\title{Furstenberg boundary of minimal actions}

\author{Zahra Naghavi}
\address{Department of Mathematics, Faculty of Mathematical Sciences, Tarbiat Modares University, Tehran 14115-134, Iran}
\email{z.naghavi@modares.ac.ir, naghavi.zahra@gmail.com}%
\subjclass[2010]{46L35, 37A55} \keywords{countable discrete group, $\Gamma$-injective envelope, Furstenberg boundary, universal minimal $\Gamma$-space, exact group}
\begin{document}


\begin{abstract}
For a countable discrete group $\Gamma$ and a minimal $\Gamma$-space $X$, we study the notion of $(\Gamma, X)$-boundary, which is a natural generalization of the notion of topological $\Gamma$-boundary in the sense of Furstenberg. We give characterizations of the $(\Gamma, X)$-boundary in terms of essential or proximal extensions. The characterization is used to answer a problem of Hadwin and Paulsen in negative. As an application, we find necessary and sufficient condition for the corresponding reduced crossed product to be exact.
\end{abstract}
\maketitle
\section{Introduction}
The notion of (topological) $\Gamma$-boundaries of a group $\Gamma$ were introduced in the 60's by Furstenberg \cite{f63}. This notion was used to be considered  as a tool to study rigidity problems in the context of semisimple Lie groups. The pioneering work of Kalantar and Kennedy \cite{kk14}, showed the key role of Furstenburg boundary in certain problems in operator algebras (see also, \cite{bkko17} and \cite{lb17}).

There are several natural generalizations of the notion of Furstenberg boundary, including that of  Bearden and Kalantar \cite{bk19}, Monod \cite{m19}, Kennedy and Schafhauser \cite{ks17}, Amini and Behrouzi \cite{ab16} and Borys \cite{b19}. In this paper, we introduce a dynamical version of the boundary for minimal  actions on compact spaces. This is essential when one deals with the notion of minimality in dynamical setting \cite{h85}, \cite{hp11}.

Throughout this paper $\Gamma$ is a countable discrete group, unless otherwise stated. Let $Y$ be a $\Gamma$-boundary, then by a result of Kalantar and Kennedy  $C(Y)$ can be considered as a $\Gamma$-essential extension of $\mathbb{C}$ \cite{kk14}. This especially tells us that any $\Gamma$-boundary can be observed as a boundary of the trivial $\Gamma$-space. On the other hand, the notion of $\Gamma$-boundary, as a minimal strongly proximal $\Gamma$-space, can be extended to the notion of minimal strongly proximal extension of a $\Gamma$-space. The latter is introduced by Glasner in \cite{g75}. In this paper, we generalize the notion of $\Gamma$-boundary through the following characterization (c.f., Theorem \ref{a}).

\vspace{.3cm}
{\bf Theorem A.}
For a countable discrete group $\Gamma$, let $X$ be a minimal $\Gamma$-space and $(Y, \varphi)$ be an extension of $X$, inducing an extension $(C(Y), \tilde{\varphi})$ of $C(X)$. The following are equivalent:
\begin{enumerate}
	\item $(C(Y), \tilde{\varphi})$ is a $\Gamma$-essential extension of $C(X)$.
	\item $Y$ is minimal and, for every $\nu\in \mathrm{Prob}(Y)$, if the restriction of Poisson map $\mathcal{P}_{\nu} : C(Y)\rightarrow\ell^{\infty}(\Gamma)$ to $C(X)$ via $\tilde{\varphi}$ is isometric, then $\mathcal{P}_{\nu}$ is isometric on $C(Y)$.
	\item $Y$ is minimal and for every $\nu\in \mathrm{Prob}(Y)$, if the push forward of $\nu$  on $X$ via $\varphi$ is contractible, then $\nu$ is contractible.
	\item $(Y, \varphi)$ is a minimal strongly proximal extension of $X$.
\end{enumerate}

We say that $(Y, \varphi)$ is a \emph{$(\Gamma, X)$-boundary} (or an \emph{$X$-boundary} for short) if it satisfies any of the above equivalent conditions.

Though we focus on minimal $\Gamma$-spaces, the above theorem (excluding the last item) holds for arbitrary $\Gamma$-spaces as well (dropping the minimality  assumption on  $Y$ in  $(2)$ and $(3)$).  We also point out that our definition of a $(\Gamma, X)$-boundary is equivalent to the definition proposed by Kennedy and Schafhauser (\cite[Remark 2.3, Corollary 2.7]{ks17}).

To make the construction of $(\Gamma, X)$-boundaries somewhat clearer, we completely describe them when $X$ is minimal and finite. Indeed we show that for a minimal  finite $\Gamma$-space $X$, any $(\Gamma, X)$-boundary can be characterized by  induced action of some $\Lambda$-boundary, where $\Lambda$ is a subgroup of $\Gamma$ of finite index  (cf, Theorem \ref{39}).

\vspace{.2cm}
{\bf Theorem B.}
Let $\Gamma$ be a countable discrete group, and let $\Lambda$ be a finite index subgroup of $\Gamma$. If $Y$ is a $\Lambda$-boundary, the induced $\Gamma$-space $\tilde{Y}$ is a $(\Gamma, \Gamma/\Lambda)$-boundary. Conversely, for a minimal finite $\Gamma$-space $X$, every $(\Gamma, X)$-boundary is the induced $\Gamma$-space of a $\Lambda$-boundary, for some finite index subgroup $\Lambda\leq\Gamma$. In particular, when $X$ is finite, the universal $(\Gamma, X)$-boundary $\partial_{F}(\Gamma, X)$ is the induced $\Gamma$-space of the Furstenberg boundary $\partial_{F}\Lambda$, for some subgroup $\Lambda\leq\Gamma$ of finite index.
\vspace{.2cm}

Hadwin and Paulsen in \cite{hp11} asked the following question: Let $\Gamma$ be a countable discrete group and $X$ be a minimal $\Gamma$-space. For the universal minimal $\Gamma$-space $L$, is $C(L)$ the $\Gamma$-injective envelope of $C(X)$? Using the notion of $(\Gamma, X)$-boundary, we give a  negative answer to this question as follows (c.f., Theorem \ref{c}).

\vspace{.2cm}
{\bf Theorem C.}
If $L$ is the universal minimal $\Gamma$-sapce and $Y$ is an $X$-boundary for a minimal finite $\Gamma$-space $X$, then $I_{\Gamma}(C(Y))\ncong C(L)$.
\vspace{.2cm}

Finally we use $(\Gamma, X)$-boundaries to study the problem of exactness for  the corresponding reduced crossed product $C(X)\rtimes_{r}\Gamma$. A discrete group $\Gamma$ is exact if the reduced group $C^{*}$-algebra $C^{*}_{r}(\Gamma)$ is exact. We  show that if $\Gamma$ is exact and $C(X)$ is $\Gamma$-injective, the action $\Gamma\curvearrowright X$ is amenable. In particular,  for the universal $(\Gamma, X)$-boundary $\partial_{F}(\Gamma, X)$, we have  the following (c.f., Theorem \ref{b}).

\vspace{.2cm}
{\bf Theorem D.}
Let $\Gamma$ be a countable discrete group. The following  are equivalent:
\begin{enumerate}
\item[(1)] $\Gamma$ is exact,
\item[(2)] For every minimal $\Gamma$-space $X$, the $\Gamma$-action on $\partial_{F}(\Gamma, X)$ is amenable,
\item[(3)] For every minimal $\Gamma$-space $X$, $C(\partial_{F}(\Gamma, X))\rtimes_{r}\Gamma$ is nuclear,
\item[(4)] For every minimal $\Gamma$-space $X$, $C(X)\rtimes_{r}\Gamma$ is exact.
\end{enumerate}
\vspace{.2cm}

The paper is organized as follows. In addition to this introduction, we have four other sections. In Section 2, we briefly review the background material. In Section 3, we discuss topological $\Gamma$-boundaries and introduce the notion of $(\Gamma, X)$-boundary for a $\Gamma$-space $X$. We show that the $(\Gamma, X)$-boundaries are the same as minimal strongly proximal extensions of $X$. This is employed to deal with the $(\Gamma, X)$-boundaries of finite $\Gamma$-space $X$, which in turn provides a negative answer to the Hadwin-Paulsen problem in Section 4. In section 5, we find conditions for the exactness of the reduced crossed product $C(X)\rtimes_{r}\Gamma$.

\subsection*{Acknowledgements}
The author is grateful to Mehrdad Kalantar for showing her the problem of Hadwin and Paulsen, as well as many helpful discussions which improved the exposition of this paper. Most of this work was completed when the author was visiting University of Houston. She would like to thank Mehrdad Kalantar for invitation and Department of Mathematics of University of Houston for  warm hospitality. She is deeply indebted to Massoud Amini, Tattwamasi Amrutam and David Kerr for valuable comments and discussions.

\section{Preliminaries}

Let $\Gamma$ be a discrete group. A compact Hausdorff space $X$ is a \emph{$\Gamma$-space} if there is a group homomorphism from $\Gamma$ into the group of homeomorphisms on $X$. In this case we write $\Gamma\curvearrowright X$. For $s\in \Gamma$ and $x\in X$ we  denote   the image of $x$ under $s$ by $sx$. The action $\Gamma\curvearrowright X$ induces an action on the algebra $C(X)$ of continuous functions on $X$ given by
$$(sf)(x)=f(s^{-1}x),\ \ \ (s\in \Gamma,\ f\in C(X), \ x\in X).$$
Similarly, $\Gamma$ acts on the set $\mathrm{Prob}(X)$ of probability measures on X via
$$s\nu (Y)=\nu (s^{-1}Y),\ \ \ (s\in \Gamma,\ \nu \in \mathrm{Prob}(X), \ Y\in\mathcal B_X).$$
A map $\varphi : Y\rightarrow X$ between $\Gamma$-spaces is a \emph{$\Gamma$-map} when $\varphi$ is continuous and $\varphi(sy)=s\varphi(y)$, for each $y\in Y$ and $s\in\Gamma$. If $\varphi : Y\rightarrow X$ is a surjective $\Gamma$-map, the pair $(Y, \varphi)$ is called an \emph{extension} (or, in some texts, a cover) of $X$. We also refer to $Y$ or $\varphi$ as an extension of $X$.

A $\Gamma$-space $Y$ is \emph{minimal} if for every $y\in Y$, the $\Gamma$-orbit $\Gamma y=\{ ty \mid t\in \Gamma\}$ is dense in $Y$, and \emph{strongly proximal} if for every probability measure $\nu\in \mathrm{Prob}(Y)$, the weak* closure of the $\Gamma$-orbit $\Gamma\nu=\{s\nu : s\in\Gamma\}$ contains a point mass $\delta_{y}\in \mathrm{Prob}(Y)$, for some $y\in Y$. A $\Gamma$-space $Y$ is said to be a \emph{$\Gamma$-boundary} if $Y$ is both minimal and strongly proximal. Furstenberg in \cite{f73} proved that every group $\Gamma$ has a unique $\Gamma$-boundary $\partial _{F}\Gamma$, which is universal, in the sense that every $\Gamma$-boundary is an image of $\partial _{F}\Gamma$.

Consider the Stone-\v{C}ech compactification $\beta \Gamma$ of $\Gamma$. The action $\Gamma\curvearrowright\beta\Gamma$ induces a semigroup structure on $\beta \Gamma$. A subset $I$ of $\beta \Gamma$ is a  {\it left ideal} if $(\beta \Gamma)I\subseteq I$. By Zorn lemma,  $\beta \Gamma$ has a minimal left ideal which is unique up to homeomorphism \cite[2.9]{hs98}. We here denote this  $\Gamma$-space by $L$. It is known that $L$ is the universal minimal $\Gamma$-space \cite[I.4]{g76}, i.e., every minimal $\Gamma$-space is an image of $L$ through a surjective $\Gamma$-map. In addition $L$ is $\Gamma$-projective \cite[3.17]{hp11}, in the sense that, for any $\Gamma$-spaces $X$ and $Y$,  any $\Gamma$-map $\psi : L\rightarrow X$, and any surjective $\Gamma$-map $\varphi : Y\rightarrow X$, there exists a $\Gamma$-map $\theta : L\rightarrow Y$ such that $\varphi\theta=\psi$.

An \emph{operator system} $\mathcal{V}$ is a unital self-adjoint subspace of a unital $C^{*}$-algebra. We say that $\mathcal{V}$ is a \emph{$\Gamma$-operator system} if there is a homomorphism from $\Gamma$ into the group of order isomorphisms of $\mathcal{V}$. A linear map $\phi : \mathcal{V}\rightarrow\mathcal{W}$ between $\Gamma$-operator systems is \emph{unital} if $\phi(1_{\mathcal{V}})=1_{\mathcal{W}}$, it is \emph{positive} if it sends positive elements to positive elements, and \emph{completely positive} (\emph{completely isometric}) if the maps $id\otimes\phi : M_{n}(\mathbb{C})\otimes\mathcal{V}\rightarrow M_{n}(\mathbb{C})\otimes\mathcal{W}$ are positive (isometric), for all $n\in\mathbb{N}$. We call $\phi : \mathcal{V}\rightarrow\mathcal{W}$ a \emph{$\Gamma$-map} if it is unital completely positive and $\Gamma$-equivariant, that is $\phi(sv)=s\phi(v)$, for each $s\in\Gamma$ and $v\in\mathcal{V}$. If a $\Gamma$-map $\phi: \mathcal{V}\rightarrow\mathcal{W}$ is completely isometric, the pair $(\mathcal{W}, \phi)$ is called an \emph{extension} of $\mathcal{V}$. In this case, we also refer to $\mathcal{W}$ or $\phi$ as an extension of $\mathcal{V}$. An extension $(\mathcal{W}, \phi)$ of $\mathcal{V}$ is \emph{$\Gamma$-essential} if for every $\Gamma$-map $\psi : \mathcal{W}\rightarrow\mathcal{U}$ such that $\psi\phi$ is completely isometric on $\mathcal{V}$, $\psi$ is completely isometric on $\mathcal{W}$. It is \emph{$\Gamma$-rigid} if for every $\Gamma$-map $\psi : \mathcal{W}\rightarrow\mathcal{W}$ such that $\psi\phi=\phi$ on $\mathcal{V}$, $\psi$ is the identity map on $\mathcal{W}$.

A $\Gamma$-operator system $\mathcal{U}$ is  \emph{$\Gamma$-injective} if for every $\Gamma$-map $\phi : \mathcal{V}\longrightarrow\mathcal{U}$ and every extension $\iota :\mathcal{V}\longrightarrow\mathcal{W}$, there is a $\Gamma$-map $\psi : \mathcal{W}\longrightarrow\mathcal{U}$ such that $\psi\iota =\phi$. Given a $\Gamma$-operator system $\mathcal{V}$, we say that $(\mathcal{I},\kappa)$ is the \emph{$\Gamma$-injective envelope} of $\mathcal{V}$ provided that $\mathcal{I}$ is $\Gamma$-injective and $(\mathcal{I}, \kappa)$ is an extension of $\mathcal{V}$ such that for any other $\Gamma$-injective $\Gamma$-operator system $\mathcal{I}_{1}$ with $\kappa(\mathcal{V})\subseteq \mathcal{I}_{1}\subseteq \mathcal{I}$, we have $\mathcal{I}_{1} = \mathcal{I}$. In the other words, $(\mathcal{I}, \kappa)$ is the $\Gamma$-injective envelope of $\mathcal{V}$ if $(\mathcal{I}, \kappa)$ is a minimal $\Gamma$-injective extension of $\mathcal{V}$. Hamana showed that for a discrete group $\Gamma$ and a $\Gamma$-operator system $\mathcal{V}$, the $\Gamma$-injective envelope always exists and is unique up to complete isometric $\Gamma$-equivariant isomorphism. In addition, the $\Gamma$-injective envelope is exactly the maximal $\Gamma$-essential extension, and it is $\Gamma$-rigid \cite[2.4, 2.6]{h85}. We denote the $\Gamma$-injective envelope of a $\Gamma$-operator system $\mathcal{V}$ by $I_{\Gamma}(\mathcal{V})$. If $\mathcal{A}$ is a $\Gamma$-$C^{*}$-algebra, an extension of $\mathcal{A}$ is a $\Gamma$-$C^{*}$-algebra containing $\mathcal{A}$ through a $\Gamma$-equivariant $*$-monomorphism. In this case, $I_{\Gamma}(\mathcal{A})$, which is the minimal $\Gamma$-injective extension of $\mathcal{A}$ in category of $\Gamma$-operator systems, is a $\Gamma$-$C^{*}$-algebra under the Choi-Effros product \cite{ce77}. In particular, for a $\Gamma$-space $X$, $I_{\Gamma}(C(X))$ is a commutative $\Gamma$-$C^{*}$-algebra. 

If $X$ is a $\Gamma$-space, for every $\nu\in \mathrm{Prob}(X)$, the $\Gamma$-map $\mathcal{P}_{\nu} :C(X)\rightarrow \ell^{\infty}(\Gamma)$, called the \emph{Poisson map}, is defined as follows:

$$\mathcal{P}_{\nu}(f)(s)=\langle f, s\nu \rangle=\int_{X}f(sx)d\nu (x),\ \ (s\in\Gamma, \ f\in C(X)).$$

Every $\Gamma$-map $\varphi:C(X)\rightarrow \ell^{\infty}(\Gamma)$ is a Poisson map for some probability measure on $X$: for $\nu=\tilde{\varphi}(\delta_{e})$,  where $\tilde{\varphi}$ is the adjoint of $\varphi$ and $\delta_{e}$ is the Dirac measure at the identity element of $\Gamma$, we have $\varphi=\mathcal{P}_{\nu}$.

\section{boundary extensions}

Kalantar and Kennedy proved in \cite{kk14} that $I_{\Gamma}(\mathbb{C})$, as the maximal $\Gamma$-essential extension of $\mathbb{C}$, can be identified with $C(\partial_{F}\Gamma)$, for the universal $\Gamma$-boundary $\partial_{F}\Gamma$. In particular,  a $\Gamma$-space $Y$ is a $\Gamma$-boundary precisely when $C(Y)$ is a $\Gamma$-essential extension of $\mathbb{C}$. We wish to replace $\mathbb{C}$ by $C(X)$, for a minimal $\Gamma$-space $X$. For this, let us first review the construction of Kalantar and Kennedy.

For a $\Gamma$-space $Y$, $C(Y)$ is a $\Gamma$-essential extension of $\mathbb{C}$, if for any $\Gamma$-operator system $\mathcal{V}$, every $\Gamma$-map $\theta : C(Y)\rightarrow\mathcal{V}$ is isometric. Let us note that one need to  verify this only for the $\Gamma$-operator system $\ell^{\infty}(\Gamma)$: if $\tau$ is in the state space of $\mathcal{V}$, there is a $\Gamma$-map $\mathcal{P}_{\tau} : \mathcal{V}\rightarrow \ell^{\infty}(\Gamma)$, given by $\mathcal{P}_{\tau}(v)(t)=\langle v, t\tau\rangle$, and $\mathcal{P}_{\tau}\theta$ is isometric. On the other hand, every $\Gamma$-map from $C(Y)$ into $\ell^{\infty}(\Gamma)$ is a Poisson map for some probability measure on $Y$. Hence $C(Y)$ is a $\Gamma$-essential extension of $\mathbb{C}$ exactly when for every $\nu\in\mathrm{Prob}(Y)$, every Poisson map $\mathcal{P}_{\nu} : C(Y)\rightarrow\ell^{\infty}(\Gamma)$ is isometric. In \cite[Th\'eor\`eme I.2]{a70}, Azencott showed that for a measure $\nu\in \mathrm{Prob}(Y)$, the Poisson map $\mathcal{P}_{\nu} : C(Y) \rightarrow \ell^{\infty}(\Gamma)$ is isometric if and only if $\{\delta_{y} : y\in Y\}\subseteq \overline{\Gamma\nu}$ in weak* topology. If a measure has this property, we say that it is \emph{contractible}. Note that if $Y$ is minimal, all measures in $\mathrm{Prob}(Y)$ are contractible precisely when $Y$ is strongly proximal. This is to say that $C(Y)$ is a $\Gamma$-essential extension of $\mathbb{C}$ if and only if $Y$ is a $\Gamma$-boundary. In particular, by considering the contravariant functor between $\Gamma$-spaces and commutative $\Gamma$-$C^{*}$-algebras, the maximal $\Gamma$-essential extension of $\mathbb{C}$ is $C(\partial_{F}\Gamma)$.

Next definition is due to Glasner \cite[page 163]{g75}.

\begin{definition}
 	Let $X$ be a $\Gamma$-space and $\varphi : Y\rightarrow X$ be an extension of $X$.
 \begin{enumerate}
 		\item[(i)] $(Y, \varphi)$ is called a \emph{minimal} extension if $Y$ is minimal.
 		\item[(ii)] $(Y, \varphi)$ is called a \emph{strongly proximal} extension if for every $\nu\in \mathrm{Prob}(Y)$ with $\textrm{supp}(\nu)\subseteq \varphi^{-1}(x)$, for some $x\in X$, $\delta_{y}\in\overline{\Gamma\nu}$, for some $y\in Y$.
 \end{enumerate}
\end{definition}

 	When $X$ is singleton with trivial action, the minimal strongly proximal extensions of $X$ are exactly topological boundaries in the sense of Furstenberg.
Following the above observations, we are lead to introduce a generalization of the notion of topological $\Gamma$-boundaries.
\begin{theorem}\label{a}
For a countable discrete group $\Gamma$, let $X$ be a minimal $\Gamma$-space and $(Y, \varphi)$ be an extension of $X$, inducing an extension $(C(Y), \tilde{\varphi})$ of $C(X)$. The following are equivalent:
\begin{enumerate}
	\item $(C(Y), \tilde{\varphi})$ is a $\Gamma$-essential extension of $C(X)$.
	\item $Y$ is minimal and, for every $\nu\in \mathrm{Prob}(Y)$, if the restriction of Poisson map $\mathcal{P}_{\nu} : C(Y)\rightarrow\ell^{\infty}(\Gamma)$ to $C(X)$ via $\tilde{\varphi}$ is isometric, then $\mathcal{P}_{\nu}$ is isometric on $C(Y)$.
	\item $Y$ is minimal and for every $\nu\in \mathrm{Prob}(Y)$, if the push forward of $\nu$  on $X$ via $\varphi$ is contractible, then $\nu$ is contractible.
	\item $(Y, \varphi)$ is a minimal strongly proximal extension of $X$.
\end{enumerate}
\end{theorem}
\begin{proof}
$(1)\Rightarrow (2)$. First we show that $Y$ is minimal. Let $L$ be the universal minimal $\Gamma$-space. Since $X$ is minimal, there is an extension $\psi : L\rightarrow X$, inducing  a $\Gamma$-equivariant $*$-monomorphism $\tilde{\psi} : C(X)\rightarrow C(L)$. The $\Gamma$-space $L$ is $\Gamma$-projective, and so there is a $\Gamma$-map $\pi : L\rightarrow Y$ with $\varphi\pi=\psi$. Since $\tilde{\pi}\tilde{\varphi}=\tilde{\psi}$ is a $*$-monomorphism and $C(Y)$ is $\Gamma$-essential, $\tilde{\pi}$ is a $*$-monomorphism. This means that $\pi$ is surjective. Thus $Y$ is a minimal $\Gamma$-space. Now $(2)$ follows, because every $\Gamma$-map from $C(Y)$ to $\ell^{\infty}(\Gamma)$ is a Poisson map, for some $\nu\in \mathrm{Prob}(Y)$.

$(2)\Rightarrow (1)$. Let $\mathcal{V}$ be a $\Gamma$-operator system and $\theta : C(Y)\rightarrow \mathcal{V}$ be a $\Gamma$-map such that the restriction of $\theta$ to $C(X)$ is isometric. Each $\tau$ in the state space of $\mathcal{V}$ induces a $\Gamma$-map $\mathcal{P}_{\tau} : \mathcal{V}\rightarrow \ell^{\infty}(\Gamma)$,  by $\mathcal{P}_{\tau}(v)(t)=\langle v, t\tau\rangle$. Choose a point mass $\delta_{x}\in \mathrm{Prob}(X)$ and extend it to a state $\tau$ on $\mathcal{V}$. Note that, since $C(X)$ is isometrically embedded into $\mathcal{V}$, sending the unit in $C(X)$  to that in $\mathcal{V}$, the positive extension is possible. But every $\Gamma$-map from $C(X)$ to $\ell^{\infty}(\Gamma)$ is a Poisson map, and a Poisson map is isometric if and only if the corresponding measure  is contractible. By the construction of $\tau$, $\mathcal{P}_{\tau}\theta|_{C(X)}$ is a Poisson map with measure $\delta_{x}$. Since $\delta_{x}$ is contractible, $\mathcal{P}_{\tau}\theta|_{C(X)}$ is an isometry. By $(2)$,  $\mathcal{P}_{\tau}\theta$ is also isometric. Therefore, for $f\in C(Y)$,
$$\|f\|=\|\mathcal{P}_{\tau}\theta(f)\|\leqslant \|\theta(f)\|\leqslant \|f\|,$$
which means that $\theta$ is isometric.

$(2)\Longleftrightarrow (3)$. It is straightforward to see that the restrictions of $\mathcal{P}_{\nu}$ to $C(X)$ is a Poisson map, with the push forward of $\nu$ as its measure. Now apply the result of Azencott.

$(3)\Rightarrow (4)$. Let $\varphi_{*}\nu$ be the push forward of $\nu$ under $\varphi$. Let us observe that $\varphi_{*}\nu =\delta_{x}$, for some $x\in X$, if and only if, $\textrm{supp}(\nu) \subseteq \varphi^{-1}(x)$. For any Borel set $E\subseteq\mathcal{B}_{X}$, if $x\in E$ then $\varphi^{-1}(x)\subseteq\varphi^{-1}(E)$. Thus $\textrm{supp}(\nu)\subseteq \varphi^{-1}(x)$ implies $\varphi_{*}\nu =\delta_{x}$. Conversely, suppose $\varphi_{*}\nu =\delta_{x}$, for some $x\in X$, and $y\notin\varphi^{-1}(x)$. Since $\varphi (y)\neq x$, there are open neighborhoods $V_{\varphi(y)}$ and $V_{x}$ such that $V_{\varphi(y)}\cap V_{x}=\emptyset$. Put $U_{y}=\varphi^{-1}(V_{\varphi(y)})$. Then $U_{y}$ is an open neighborhood of $y$ such that $\nu(U_{y})=\nu(\varphi^{-1}(V_{\varphi(y)}))=0$. Hence $y\notin \textrm{supp}(\nu)$.
Now if $\textrm{supp}(\nu)\subseteq \varphi^{-1}(x)$, for some $x$, $(4)$ follows because $\varphi_{*}\nu =\delta_{x}$ is contractible.

$(4)\Rightarrow (3)$. Suppose $\nu\in\mathrm{Prob}(Y)$ such that $\varphi_{*}\nu$ is contractible. Then $\delta_{x}\in\overline{\Gamma\varphi_{*}\nu}$, for some $x\in X$. Since $\varphi_{*}$ is isometric, this is equivalent to the existence of $\nu^{\prime}\in \mathrm{Prob}(Y)$ with $\varphi_{*}\nu^{\prime}=\delta_{x}$, and $\nu^{\prime}\in \overline{\Gamma\nu}$. This is to say that there exists  $\nu^{\prime}\in \mathrm{Prob}(Y)$ such that $\textrm{supp}(\nu^{\prime})\subseteq\varphi^{-1}(x)$ and $\nu^{\prime}\in \overline{\Gamma\nu}$. By $(3)$, $\delta_{y}\in \overline{\Gamma\nu^{\prime}}$, for some $y\in Y$. Since $\overline{\Gamma\nu^{\prime}}\subseteq\overline{\Gamma\nu}$, we get $\delta_{y}\in\overline{\Gamma\nu}$. This plus minimality of $Y$ finishes the proof.
\end{proof}

\begin{definition}\label{def1}
We say that $(Y, \varphi)$ is a \emph{$(\Gamma, X)$-boundary} (or simply a \emph{$X$-boundary}), if $(Y,\varphi)$ satisfies any of the above equivalent conditions.
\end{definition}

When $X$ is singleton with trivial action, the $X$-boundaries are exactly the topological boundaries in the sense of Furstenberg.

For a minimal $\Gamma$-space $X$, the commutative $C^{*}$-algebra $I_{\Gamma}(C(X))$ is the maximal $\Gamma$-essential extension of $C(X)$. By Definition \ref{def1}, the spectrum of $I_{\Gamma}(C(X))$ is a $(\Gamma, X)$-boundary. We denote this $\Gamma$-space, which is unique up to homeomorphism, by $\partial_{F}(\Gamma, X)$ and write $I_{\Gamma}(C(X))=C(\partial_{F}(\Gamma, X))$. Let us show that $\partial_{F}(\Gamma, X)$ is the universal $(\Gamma, X)$-boundary. Suppose $(Y, \varphi)$ is a $(\Gamma, X)$-boundary, inducing an extension $(C(Y), \tilde{\varphi})$ of $C(X)$, and let $\tilde{\pi}: C(X)\rightarrow C(\partial_{F}(\Gamma, X))$ be an extension of $C(X)$ with respect to it $(\partial_{F}(\Gamma, X), \pi)$ is a $(\Gamma, X)$-boundary. Since $C(\partial_{F}(\Gamma, X)$ is $\Gamma$-injective and $C(Y)$ is $\Gamma$-essential, there exists an injective $\Gamma$-map $\alpha : C(Y)\rightarrow C(\partial_{F}(\Gamma, X))$ such that $\alpha\tilde{\varphi}=\tilde{\pi}$. Note that this map is not necessarily $*$-homomorphism. Consider the adjoint maps $\tilde{\tilde{\pi}}: \mathcal{M}(\partial_{F}(\Gamma, X))\rightarrow \mathcal{M}(X)$, $\tilde{\tilde{\varphi}}: \mathcal{M}(Y)\rightarrow \mathcal{M}(X)$ and $\tilde{\alpha}: \mathcal{M}(\partial_{F}(\Gamma, X))\rightarrow \mathcal{M}(Y)$ between the spaces of regular Borel measures on these $\Gamma$-spaces, and note that $\tilde{\tilde{\varphi}}\tilde{\alpha}=\tilde{\tilde{\pi}}$ and $\tilde{\tilde{\pi}}|_{\partial_{F}(\Gamma, X)}=\pi$. Let $\nu=\tilde{\alpha}(\delta_{f})$ for $f\in \partial_{F}(\Gamma, X)$. Then $\tilde{\tilde{\varphi}}(\nu)=\delta_{x}$ for some $x\in X$, which means $\textrm{supp}(\nu)\subseteq\varphi^{-1}(x)$. Now $Y$ is a $(\Gamma, X)$-boundary, so $\{\delta_{y} : y\in Y\}\subseteq\tilde{\alpha}(\{\delta_{f} : f\in \partial_{F}(\Gamma, X)\})$. Since $\partial_{F}(\Gamma, X)$ is minimal, $\tilde{\alpha}(\{\delta_{f} : f\in \partial_{F}(\Gamma, X)\})=\{\delta_{y} : y\in Y\}$. It means there exists a surjective $\Gamma$-map from $\partial_{F}(\Gamma, X)$ onto $Y$. So $\partial_{F}(\Gamma, X)$ is the universal $(\Gamma, X)$-boundary. In particular, the universal strongly proximal extension of a minimal $\Gamma$-space $X$ always exists. We mention that if $X$ is singleton, $\partial_{F}(\Gamma, X)$ is nothing but the Furstenberg universal $\Gamma$-boundary $\partial_{F}\Gamma$.

Next we investigate the structure of $(\Gamma, X)$-boundaries when $X$ is a finite minimal  $\Gamma$-space.

	Let $\Gamma\curvearrowright X$ is an action of  $\Gamma$ on  $X$ and $\Lambda$ be another  discrete group. A \emph{cocycle} of the action in $\Lambda$ is a map $\alpha : \Gamma\times X\rightarrow \Lambda$ such that
	$$\alpha(\gamma_{1}\gamma_{2}, x)=\alpha(\gamma_{1}, \gamma_{2} x)\alpha(\gamma_{2}, x),$$
	for all $\gamma_{1},\gamma_{2}\in \Gamma$ and $x\in X$.

We need the notion of induced $\Gamma$-spaces \cite[4.2.21]{z84} (c.f., \cite[2.2.4]{cp12}).
Let $\Lambda$ be a finite index subgroup of a countable discrete group $\Gamma$,  $Y$ be a $\Lambda$-space, and $\tilde{Y}=\Gamma/\Lambda\times Y$. Take a transversal $T=\{t_{1},\ldots, t_{n}\}$ for $\Gamma/\Lambda$ such that $t_{1}=e$. Define the cocycle $\alpha : \Gamma\times \Gamma/\Lambda\rightarrow \Lambda$ by $\alpha(\gamma , t_{i}\Lambda)=\lambda$, such that $\gamma t_{i}\lambda\in T$, and observe that such a $\lambda$ is unique. Now $\Gamma$ acts on $\Gamma/\Lambda\times Y$ by
$$\gamma . (t_{i}\Lambda, y)=(\gamma t_{i}\alpha(\gamma, t_{i}\Lambda)\Lambda,\alpha(\gamma, t_{i}\Lambda)^{-1}y), \ \ \ \ (\gamma\in\Gamma, y\in Y).$$
The $\Gamma$-space $\tilde{Y}$ is called the \emph{induced}  $\Gamma$-space of the $\Lambda$-space $Y$.

\begin{theorem} \label{39}
Let $\Gamma$ be a countable discrete group, and let $\Lambda$ be a finite index subgroup of $\Gamma$. If $Y$ is a $\Lambda$-boundary, the induced $\Gamma$-space $\tilde{Y}$ is a $(\Gamma, \Gamma/\Lambda)$-boundary. Conversely, for a minimal finite $\Gamma$-space $X$, every $(\Gamma, X)$-boundary is the induced $\Gamma$-space of a $\Lambda$-boundary, for some finite index subgroup $\Lambda\leq\Gamma$. In particular, when $X$ is finite, the universal $(\Gamma, X)$-boundary $\partial_{F}(\Gamma, X)$ is the induced $\Gamma$-space of the Furstenberg boundary $\partial_{F}\Lambda$, for some subgroup $\Lambda\leq\Gamma$ of finite index.
\end{theorem}
\begin{proof}
With the above notations, we show that $\tilde{Y}$ is a $(\Gamma, \Gamma/\Lambda)$-boundary. To see that $\tilde{Y}$ is $\Gamma$-minimal, let $t\in T$ and $y,y^{\prime}\in Y$. Since $Y$ is $\Lambda$-minimal, there exists $\{\lambda_{\ell}\}\subseteq \Lambda$ such that $\lambda_{\ell}y\rightarrow y^{\prime}$. Fix $\lambda_{\ell}\in \Lambda$, and observe that $\alpha(t\lambda_{\ell}, \Lambda)=\lambda_{\ell}^{-1}$. Thus
$$t\lambda_{\ell}(\Lambda, y)=(t\lambda_{\ell}\alpha(t\lambda_{\ell}, \Lambda)\Lambda, \alpha(t\lambda_{\ell}, \Lambda)^{-1}y)=(t\Lambda, \lambda_{\ell}y).$$
When $\ell$ tends to infinity,  $t\lambda_{\ell}(\Lambda, y)\rightarrow(t\Lambda, y^{\prime})$. Therefore, $\overline{\Gamma(\Lambda, y)}=\tilde{Y}$. Similarly, $\lambda_{\ell}t^{-1}(t\Lambda, y)\rightarrow (\Lambda, y^{\prime})$, which implies that $(\Lambda , y^{\prime})\in \overline{\Gamma (t\Lambda,y)}$, for $t\in T$. Hence $\overline{\Gamma(\Lambda, y^{\prime})}\subseteq\overline{\Gamma(t\Lambda, y)}$. We have shown that
$$\tilde{Y}=\overline{\Gamma (\Lambda, y^{\prime})}\subseteq\overline{\Gamma (t\Lambda, y)}\subseteq \tilde{Y}.$$
Therefore, $\overline{\Gamma (t\Lambda, y)}=\tilde{Y}$, for $t\in T$ and $y\in Y$.

Next let us observe that $\tilde{Y}$ is a strongly proximal extension of $\Gamma/\Lambda$. Consider the continuous surjective map $\varphi : \tilde{Y}\rightarrow \Gamma/\Lambda$ given by $\varphi(t_{i}\Lambda, y)=t_{i}\Lambda$, for $y\in Y$. We show that $\varphi$ is $\Gamma$-equivariant. Put $\Lambda_{j}=t_{j}\Lambda t_{j}^{-1}$. Then $\Lambda_{j}$ is a subgroup of $\Gamma$, and $\{t_{1}t_{j}^{-1},\ldots ,t_{n}t_{j}^{-1}\}$ is a transversal for $\Gamma /\Lambda_{j}$. Therefore, $\Gamma=\bigsqcup_{i=1}^{n}t_{i}t_{j}^{-1}\Lambda_{j}$. If $\gamma=t_{i}t_{j}^{-1}\lambda_{j}\in t_{i}t_{j}^{-1}\Lambda_{j}$,  since $t_{j}^{-1}\lambda_{j}t_{j}\in\Lambda$, we have $\gamma\varphi(t_{j}\Lambda, y)=t_{i}\Lambda$. On the other hand, $\alpha(\gamma, t_{j}\Lambda)=e$, and so
\begin{align*}
\varphi(\gamma(t_{j}\Lambda, y)) &=\varphi(\gamma t_{j}\alpha(\gamma, t_{j}\Lambda)\Lambda, \alpha(\gamma, t_{j}\Lambda)^{-1}y)  \\
                                        &=\varphi(t_{i}\Lambda,y)=t_{i}\Lambda.
\end{align*}
Thus, $\gamma\varphi(t_{j}\Lambda, y)=\varphi(\gamma(t_{j}\Lambda, y))$. We have shown that $\varphi : \tilde{Y}\rightarrow \Gamma/\Lambda$ is an extension of $\Gamma/\Lambda$. To show that the extension is strongly proximal, let $\nu\in \mathrm{Prob}(\tilde{Y})$ such that $\textrm{supp}(\nu)\subseteq\varphi^{-1}(t_{i}\Lambda)=\{t_{i}\Lambda\times Y\}$, for some $t_{i}\in T$. Then $\nu =\delta_{t_{i}\Lambda}\times\mu$, for some $\mu\in \mathrm{Prob}(Y)$. Since $Y$ is $\Lambda$-strongly proximal, there exists $\{\lambda_{\ell}\}\subseteq\Lambda$ such that $\lambda_{\ell}\mu\rightarrow\delta_{y_{0}}$, for some $y_{0}\in Y$, in weak* topology. We claim that if $\{\gamma_{\ell}\}_{\ell}=\{t_{i}\lambda_{\ell}t_{i}^{-1}\}_{\ell}\subseteq\Gamma$, then $\gamma_{\ell}\nu\rightarrow\delta_{(t_{i}\Lambda, y_{0})}$, in weak* topology. Let $h\in C(\Gamma/\Lambda\times Y)$ and fix $\lambda_{\ell}\in \Lambda$. For $t_{i}\in T$,  $\alpha(t_{i}\lambda_{\ell}t_{i}^{-1}, t_{i}\Lambda)=\lambda_{\ell}^{-1}$,
\begin{align*}
&\int h((t_{i}\lambda_{\ell}t_{i}^{-1})(t\Lambda, y))d\nu(t\Lambda, y)   \\
    &= \int h((t_{i}\lambda_{\ell}t_{i}^{-1})(t\Lambda, y))d(\delta_{t_{i}\Lambda}\times\mu)(t\Lambda, y)     \\
    &= \int h((t_{i}\lambda_{\ell}t_{i}^{-1})(t_{i}\Lambda, y))d(\delta_{t_{i}\Lambda}\times\mu)(t_{i}\Lambda, y)   \\
    &= \int h(t_{i}\lambda_{\ell}\alpha(t_{i}\lambda_{\ell}t_{i}^{-1}, t_{i}\Lambda)\Lambda,\alpha(t_{i}\lambda_{\ell}t_{i}^{-1}, t_{i}\Lambda)^{-1}y)d(\delta_{t_{i}\Lambda}\times\mu)(t_{i}\Lambda, y)    \\
    &= \int h((t_{i}\Lambda, \lambda_{\ell}y))d(\delta_{t_{i}\Lambda}\times\mu)(t_{i}\Lambda, y).
\end{align*}
When $\lambda_{\ell}$ tends to infinity,
$$\int h((t_{i}\lambda_{\ell}t_{i}^{-1})(t\Lambda, y))d\nu(t\Lambda, y)\rightarrow h(t_{i}\Lambda, y_{0}).\ \ \ \ \ \ \ \ \ \ \ \ \ \ \ \ \ \ \ \ \ $$
Therefore, $\gamma_{\ell}\nu\rightarrow\delta_{(t_{i}\Lambda, y_{0})}$. This  means that $\tilde{Y}$ is a $\Gamma$-strongly proximal extension of $X$, and so is a $(\Gamma, \Gamma/\Lambda)$-boundary.

Note that $\tilde{Y}=\bigsqcup_{i=1}^{n}t_{i}\Lambda\times Y$. It is not hard to see that any $t_{i}\Lambda\times Y$ is a $\Lambda_{i}$-boundary when $\Lambda_{i}=t_{i}\Lambda t_{i}^{-1}$.
Conversely, Let $X=\{x_{1},\ldots, x_{n}\}$ be a minimal $\Gamma$-space. Let $\varphi : Y\rightarrow X$ be a surjective $\Gamma$-map making $(Y,\varphi)$ a $(\Gamma, X)$-boundary. We write $Y=\bigsqcup_{i=1}^{n}Y_{i}$, for $Y_{i}=\varphi^{-1}(x_{i})$, $1\leq i\leq n$. Fix $x_{i}\in X$, and consider the stabilizer subgroup $\Lambda_{i}=\{\gamma\in \Gamma : \gamma x_{i}=x_{i}\}$. We claim that $Y_{i}$ is a $\Lambda_{i}$-boundary. First let us show that $Y_{i}$ is $\Lambda_{i}$-minimal. Given $y, y^{\prime}\in Y_{i}$, since $Y$ is $\Gamma$-minimal, there exists $\{\gamma_{\ell}\}\subseteq\Gamma$ such that $\gamma_{\ell}y \rightarrow y^{\prime}$. Since $y, y^{\prime}\in Y_{i}=\varphi^{-1}(x_{i})$ and $\varphi$ is a $\Gamma$-map,  $\gamma_{\ell}x_{i}\rightarrow x_{i}$. Therefore, for sufficiently large  $\ell$, $\gamma_{\ell}\in \Lambda_{i}$. Hence $Y_{i}$ is $\Lambda_{i}$-minimal. To show that $Y_{i}$ is $\Lambda_{i}$-strongly proximal,  take $\nu\in \mathrm{Prob}(Y_{i})$, then since $(Y, \varphi)$ is a $(\Gamma, X)$-boundary and $\textrm{supp}(\nu)\subseteq Y_{i}=\varphi^{-1}(x_{i})$,  there exists $y\in \varphi^{-1}(x_{i})$ and $\{\gamma_{\ell}\}\subseteq \Gamma$ such that $\gamma_{\ell}\nu\rightarrow \delta_{y}$. Thus $\gamma_{\ell}\varphi_{*}\nu=\varphi_{*}(\gamma_{\ell}\nu)\rightarrow\delta_{x_{i}}$. Also, $\textrm{supp}(\nu)\subseteq \varphi^{-1}(x_{i})$, exactly when $\varphi_{*}\nu=\delta_{x_{i}}$. This implies that $\delta_{\gamma_{\ell}x_{i}}=\gamma_{\ell}\delta_{x_{i}}\rightarrow \delta_{x_{i}}$. Thus, for sufficiently large $\ell$, $\delta_{\gamma_{\ell}x_{i}}=\delta_{x_{i}}$. Therefore, for sufficiently large $\ell$, $\gamma_{\ell}\in \Lambda_{i}$.

Without loss of generality, we suppose $i=1$. Let $y\in Y_{1}=\varphi^{-1}(x_{1})$, and let $T=\{t_{1},\ldots, t_{n}\}$ be a transversal for $\Gamma/\Lambda_{1}$. For every $1\leq i\leq n$, $t_{i}x_{1}=x_{i}$, which implies that $t_{i}y\in Y_{i}$. Thus, $t_{i}Y_{1}\subseteq Y_{i}$. Similarly, $Y_{i}\subseteq t_{i}Y_{1}$. Therefore, $Y$ is a disjoint union of $n$-copies of $Y_{1}$. By considering the homeomorphism $\bigsqcup_{t_{i}\in T}t_{i}Y_{1}\longrightarrow \Gamma/\Lambda_{1}\times Y_{1}$, given by $t_{i}y\mapsto (t_{i}\Lambda_{1}, y)$, and inducing the action of $\bigsqcup_{t_{i}\in T}t_{i}Y_{1}$ to $\Gamma/\Lambda_{1}\times Y_{1}$, the $(\Gamma, X)$-boundary $Y$ is in the form of an induced $\Gamma$-space of  $\Lambda_{1}$-space $Y_{1}$. Note that $\Gamma/\Lambda_{1}\cong X$.

To prove the last part, since $\partial_{F}(\Gamma, X)$ is a $(\Gamma, X)$-boundary,  $\partial_{F}(\Gamma, X)$ is in the form of an induced $\Gamma$-space $X\times Y$, where $Y$ is a $\Lambda$-boundary, for some finite index subgroup $\Lambda$. On the other hand, the induced $\Gamma$-space $X\times\partial_{F}\Lambda$ is a $(\Gamma, X)$-boundary. There exists a surjective $\Lambda$-map $\theta: \partial_{F}\Lambda\rightarrow Y$ which induces a surjective $\Gamma$-map $\Theta : X\times\partial_{F}\Lambda\rightarrow X\times Y$, given by $\Theta(x, f)=(x, \theta(f))$. Since $\partial_{F}(\Gamma, X)\cong X\times Y$ is universal,  $X\times Y\cong X\times\partial_{F}\Lambda$.
\end{proof}

By the above theorem,  if $\Gamma$ is amenable and $(Y, \varphi)$ is a $(\Gamma, X)$-boundary for minimal finite $\Gamma$-space $X=\{x_{1},\ldots , x_{n}\}$, then $Y$ has exactly $n$ elements. This is because $Y=\bigsqcup_{i=1}^{n}\varphi^{-1}(x_{i})$, where for every $i$, $\varphi^{-1}(x_{i})$ is a $\Lambda_{i}$-boundary for $\Lambda_{i}=\{\gamma\in\Gamma: \gamma x_{i}=x_{i}\}$. Now every $\Lambda_{i}$ is amenable, which implies that every $\varphi^{-1}(x_{i})$ is singleton.

\section{On a problem of Hadwin and Paulsen}
There is a contravariant functor between the category of compact Hausdorff spaces with continuous maps and the category of unital commutative $C^{*}$-algebras with $*$-homomorphisms, sending  projective objects to  injective objects. Hadwin and Paulsen \cite{hp11} showed that for every compact Hausdorff space $X$, there is a unique projective cover $P$, which is minimal among all projective covers of $X$. As a result, the injective envelope $I(C(X))$ of $C(X)$ is $*$-isomorphic to $C(P)$. They also extended this to the case when a countable discrete group $\Gamma$ acts on a compact space $X$, using the functor between compact $\Gamma$-spaces and  unital commutative $\Gamma$-$C^{*}$-algebras. In this case,  unlike the previous case where rigidity and essentiality of projective covers are equivalent \cite[Proposition 2.11]{hp11}, there is no $\Gamma$-projective, $\Gamma$-rigid cover, even when $X$ is a singleton \cite[Proposition 3.1]{hp11}. However, one still could work with  $\Gamma$-essential covers.

Hadwin and Paulsen showed that if $\Gamma$ is a countable discrete group and $X$ is a minimal $\Gamma$-space, a minimal left ideal of the Stone-\v{C}ech compactification of $\Gamma$, is the minimal $\Gamma$-projective cover of $X$. This leads naturally to the question that for a minimal left ideal $L$ in $\beta \Gamma$, is $C(L)$ $*$-isomorphic to $I_{\Gamma}(C(X))$? Recall that any two minimal left ideals in $\beta\Gamma$ are homeomorphic, and  the minimal left ideal in $\beta\Gamma$ is nothing but the universal minimal $\Gamma$-space.

Since $L$ is $\Gamma$-projective, $C(L)$ is $\Gamma$-injective in the category of $\Gamma$-$C^{*}$-algebras, and so $I_{\Gamma}(C(L))=C(L)$. However, as we see soon, the problem of Hadwin and Paulsen has negative answer in general. For this let us first observe that for an arbitrary countable infinite group $\Gamma$, there is an infinite compact minimal $\Gamma$-space which has an invariant measure.

A probability-measure-preserving (p.m.p.) action of a group $\Gamma$ on a probability measure space $(X, \mu )$ is a homomorphism of $\Gamma$ into the group of measure-preserving transformations on $X$, parameterized by $\Gamma$. In this context, the action $\Gamma$ on $(X, \mu )$ is said to be \emph{free} if there is a $\Gamma$-invariant set $X_{0}\subseteq X$ with $\mu (X_{0})=1$, such that if $sx=x$, for some $x\in X_{0}$ and $s\in \Gamma$, then $s=e$.

Let $\Gamma$ be a countabe infinite group. The Bernoulli shift action of $\Gamma$ on the space $\{ 0, 1 \} ^{\Gamma}$, with any invariant probability measure (for example, the product of equiprobability measure on $\{ 0, 1 \}$) is a free p.m.p. action. Thus, every countable infinite group admits at least one non-trivial infinite free p.m.p. action.  Benjamin Weiss in \cite[6.1]{w12} has shown that if $\Gamma$ is any countable infinite group and $\Gamma\curvearrowright (X, \mu)$ is any free p.m.p. action, there is a minimal continuous action as a subshift of $([0,1]\times \mathbb{N})^{\Gamma}$, which admits an invariant measure and is a model for $(X, \mu)$ (that is, an isomorphic copy of the action which is also continuous). In particular, there exists a non-trivial minimal $\Gamma$-space with an invariant measure. We note that this $\Gamma$-space is infinite.

\begin{lemma}\label{thel}
Suppose that $L$ is the universal minimal $\Gamma$-space, $X$ is a minimal $\Gamma$-space, and $Z$ is an infinite minimal $\Gamma$-space with an invariant measure $\mu$. Let $\varphi : L\rightarrow X$ and $\theta : L\rightarrow Z$ be surjective $\Gamma$-maps for which the set of all pull backs of $\mu$ under $\theta$ contains a measure $\nu$ such that $\nu(\varphi^{-1}(x))\neq 0$, for some $x\in X$, then $I_{\Gamma}(C(X))\ncong C(L)$.
\end{lemma}
\begin{proof}
Let $(C(L), \tilde{\varphi})$ be a $\Gamma$-essential extension of $C(X)$ induced by a surjective $\Gamma$-map $\varphi : L\rightarrow X$ and let $\theta : L\rightarrow Z$ be a surjective $\Gamma$-map. Fix $\varphi$ and $\theta$, and suppose $\nu$ is a pull back of $\mu$ under $\theta$ such that $\nu(\varphi^{-1}(x))\neq 0$. Define $\nu_{x}\in \mathrm{Prob}(L)$ by $\nu_{x}(E)=\frac{\nu(E\cap\varphi^{-1}(x))}{\nu(\varphi^{-1}(x))}$. It is easy to see that $\textrm{supp}(\nu_{x})\subseteq\varphi^{-1}(x)$ and $\nu_{x}\ll\nu$. Since $L$ is a $(\Gamma, X)$-boundary, there exists $\{t_{\alpha}\}\subseteq\Gamma$ and $\ell\in L$ such that $t_{\alpha}\nu_{x}\rightarrow\delta_{\ell}$. Let $\theta_{*}\nu_{x}=\mu_{x}$, when $\theta_{*}\nu_{x}$ denotes the push forward of $\nu_{x}$ under $\theta$, then $t_{\alpha}\mu_{x}\rightarrow\delta_{\theta(\ell)}$. Moreover,  $\mu_{x}=\theta_{*}\nu_{x}\ll\theta_{*}\nu=\mu$. Thus $\mu$ is invariant, and $\delta_{\theta(\ell)}=\lim_{\alpha}t_{\alpha}\mu_{x}\ll \lim_{\alpha}t_{\alpha}\mu=\mu$. Therefore, $\mu(\{\theta(\ell)\})>0$.

Consider the orbit $\Gamma\theta(\ell)$ of $\theta(\ell)$. For $x\in \Gamma\theta(\ell)$, $\mu(\{x\})=\mu(\{\theta(\ell)\})$. In particular, $\Gamma\theta(\ell)$ must be finite, since otherwise, $\mu(Z)>\mu(\Gamma\theta(\ell))>1$. On the other hand, since $Z$ is minimal,  $\Gamma\theta(\ell)=\overline{\Gamma\theta(\ell)}=Z$. This implies that $Z$ is finite, which is a contradiction. Thus $C(L)$ could not be a $\Gamma$-essential extension of $C(X)$. So $I_{\Gamma}(C(X))\ncong C(L)$.
\end{proof}

\begin{theorem}\label{c}
If $L$ is the universal minimal $\Gamma$-sapce and $Y$ is an $X$-boundary for a minimal finite $\Gamma$-space $X$, then $I_{\Gamma}(C(Y))\ncong C(L)$.
\end{theorem}
\begin{proof}
Let us first observe that $I_{\Gamma}(C(X))\ncong C(L)$. Consider a surjective $\Gamma$-map $\varphi : L\rightarrow X$. In the notations of above lemma, let $\theta : L\rightarrow Z$ be a surjective $\Gamma$-map and $\nu$ be any member of the set of all pull backs of $\mu$ under $\theta$. Let $X=\{x_{1}, \ldots\, x_{n}\}$. Since $L=\bigcup_{i=1}^{n}\varphi^{-1}(x_{i})$, $\nu(L)=\sum_{i=1}^{n}\nu(\varphi^{-1}(x_{i}))$. We have $\nu(L)\neq 0$, hence there exists  $x\in X$ such that $\nu(\varphi^{-1}(x))\neq 0$. So by above lemma, $I_{\Gamma}(C(X))\ncong C(L)$.

On the other hand, $I_{\Gamma}(C(X))=C(\partial_{F}(\Gamma, X))$, for the universal $X$-boundary $\partial_{F}(\Gamma, X)$. Thus $C(\partial_{F}(\Gamma, X))$ is $\Gamma$-injective, which implies $I_{\Gamma}(C(\partial_{F}(\Gamma, X)))=C(\partial_{F}(\Gamma, X))$, that is, $I_{\Gamma}(C(\partial_{F}(\Gamma, X)))\neq C(L)$.

Now if $Y$ is an $X$-boundary, we have $C(X)\hookrightarrow C(Y)\hookrightarrow C(\partial_{F}(\Gamma, X))$, with $I_{\Gamma}(C(X))=I_{\Gamma}(C(\partial_{F}(\Gamma, X)))$. Thus $I_{\Gamma}(C(Y))=C(\partial_{F}(\Gamma, X))$, which means $I_{\Gamma}(C(Y))\ncong C(L)$.
\end{proof}

\section{Applications to reduced crossed products}
In this section we apply our results to find conditions for exactness of the reduced crossed product of the (minimal) action of a countable group on a compact space. For the general theory of discrete exact groups and amenable actions, we refer the reader to \cite{bo08}.

Recall that a group $\Gamma$ is exact if $C^{*}_{r}(\Gamma)$ is exact as a $C^{*}$-algebra. This is introduced by Kirchberg and Wasserman in \cite{kw99} and is known to be equivalent to the amenability of $\Gamma$ actions on arbitrary compact spaces. Ozawa observed that one needs only the amenability of  canonical action on  $\beta\Gamma$ \cite{o00}. Exactness of $\Gamma$ is also known to be equivalent to the amenability of the $\Gamma$-action on the Furstenberg boundary $\partial_{F}\Gamma$ \cite[4.5]{kk14}. In the latter case, the key point is that $C(\partial_{F}\Gamma)$ is $\Gamma$-injective. In this section we show that the same  idea could be adapted to show  that $\Gamma$ is exact if and only if the $\Gamma$-action on $\partial_{F}(\Gamma, X)$ is amenable, for every  minimal $\Gamma$-space $X$.

\begin{lemma}\label{lem8}
Suppose $\Gamma$ is exact and $X$ is a $\Gamma$-space. If $C(X)$ is $\Gamma$-injective, then the action $\Gamma\curvearrowright X$ is amenable.
\end{lemma}
\begin{proof}
Since $C(X)$ is $\Gamma$-injective, there is a $\Gamma$-map $\psi : \ell^{\infty}(\Gamma)\rightarrow C(X)$. Identifying $\ell^{\infty}(\Gamma)$ with $C(\beta\Gamma)$, restriction of the adjoint  map $\tilde{\psi}: \mathcal{M}(X)\rightarrow\mathcal{M}(\beta\Gamma)$ to the space of point masses on $X$ gives a continuous $\Gamma$-equivariant map $\theta: X\rightarrow \mathrm{Prob}(\beta\Gamma)$. Since $\Gamma$ is exact, the $\Gamma$-action on $\beta\Gamma$ is amenable and so is the $\Gamma$-action on $\mathrm{Prob}(\beta\Gamma)$ \cite[Proposition 9]{cm14} (see also, \cite[3.6]{h00}). Thanks to the existence of the $\Gamma$-equivariant map $\theta$, the $\Gamma$-action on $X$ is also amenable.
\end{proof}

Next we show that if $X$ is a minimal $\Gamma$-space, the exactness passes from $\Gamma$ to the reduced crossed product $C(X)\rtimes_{r}\Gamma$. Recall that a $C^{*}$-algebra is {\it exact} if and only if it can be embedded into a nuclear $C^{*}$-algebra \cite{k95}, \cite{w90}.

\begin{theorem}\label{b}
Let $\Gamma$ be a countable discrete group. The following  are equivalent:
\begin{enumerate}
	\item[(1)] $\Gamma$ is exact,
	\item[(2)] For every minimal $\Gamma$-space $X$, the $\Gamma$-action on $\partial_{F}(\Gamma, X)$ is amenable,
	\item[(3)] For every minimal $\Gamma$-space $X$, $C(\partial_{F}(\Gamma, X))\rtimes_{r}\Gamma$ is nuclear,
	\item[(4)] For every minimal $\Gamma$-space $X$, $C(X)\rtimes_{r}\Gamma$ is exact.
\end{enumerate}
\end{theorem}
\begin{proof}
$(1)\Rightarrow (2)$. Since $C(\partial_{F}(\Gamma, X))=I_{\Gamma}(C(X))$, $C(\partial_{F}(\Gamma, X))$ is $\Gamma$-injective. Now the result follows from Lemma \ref{lem8}.

$(2)\Rightarrow (3)$. This is well known (see, \cite[4.3.4, 4.3.7]{bo08}).

$(3)\Rightarrow (4)$. We have,
$$C(X)\rtimes_{r}\Gamma\subseteq I_{\Gamma}(C(X))\rtimes_{r}\Gamma = C(\partial_{F}(\Gamma, X))\rtimes_{r}\Gamma,$$
and by assumption, $C(\partial_{F}(\Gamma, X))\rtimes_{r}\Gamma$ is nuclear. Thus $C(X)\rtimes_{r}\Gamma$ is exact, as it is embedded into a nuclear $C^{*}$-algebra \cite[Proposition 7]{w90}.

$(4)\Rightarrow (1)$. Just let $X$ be a singleton.
\end{proof}

 Finding a tangible nuclear $C^{*}$-algebra containing a given exact $C^{*}$-algebra is the next natural thing to ask for. Ozawa has conjectured that for an exact $C^{*}$-algebra $A$, there is a nuclear $C^{*}$-algebra $\mathcal{N}(A)$ such that $A\subset\mathcal{N}(A)\subset I(A)$, where $I(A)$ is the injective envelope of $A$. Kalantar and Kennedy proved that if $\Gamma$ is a discrete exact group, for the reduced group $C^{*}$-algebra $C^{*}_{r}(\Gamma)$, there is a canonical unital nuclear $C^{*}$-algebra $\mathcal{N}(C^{*}_{r}(\Gamma))$ such that $C^{*}_{r}(\Gamma)\subset\mathcal{N}(C^{*}_{r}(\Gamma))\subset I(C^{*}_{r}(\Gamma))$ \cite[4.6]{kk14}. Indeed, they observed that $C(\partial_{F}\Gamma)\rtimes_{r}\Gamma$, which is  nuclear  when $\Gamma$ is exact, could play the role of $\mathcal{N}(C^{*}_{r}(\Gamma))$. The above result shows that the same could be done in the more general case of minimal $\Gamma$-spaces.

\begin{corollary}
Let $\Gamma$ be an exact group and let $X$ be a minimal $\Gamma$-space. There is a canonical unital nuclear $C^{*}$-algebra $\mathcal{N}(C(X)\rtimes_{r}\Gamma)$ such that
$$C(X)\rtimes_{r}\Gamma\subset\mathcal{N}(C(X)\rtimes_{r}\Gamma)\subset I(C(X)\rtimes_{r}\Gamma),$$
where $I(C(X)\rtimes_{r}\Gamma)$ is the injective envelope of $C(X)\rtimes_{r}\Gamma$.
\end{corollary}
\begin{proof}
Take $\mathcal{N}(C(X)\rtimes_{r}\Gamma)=C(\partial_{F}(\Gamma, X))\rtimes_{r}\Gamma$. Since $\Gamma$ is exact, by Theorem \ref{b}, $C(\partial_{F}(\Gamma, X))\rtimes_{r}\Gamma$ is nuclear and $C(X)\rtimes_{r}\Gamma$ is exact. The second inclusion follows from \cite[3.4]{h85}.
\end{proof}

The results in this section hold also for  arbitrary $\Gamma$-spaces. Thus (a slight modification of) the above corollary proves a recent  result of Buss, Echterhoff and Willett  \cite[Corolarry 8.4]{bew19}.


\end{document}